\documentclass{amsart}
\usepackage{amssymb,amsmath,amsthm,enumerate,amscd,graphicx,color}
\usepackage[usenames,dvipsnames,svgnames,table]{xcolor}
\usepackage{pdfsync}
 \usepackage[colorlinks=true, pdfstartview=FitV, linkcolor=blue, citecolor=blue, urlcolor=blue]{hyperref}
\newcommand{\thmref}[1]{Theorem~\ref{#1}}

\newcommand{\eqnref}[1]{~(\ref{#1})}

\def\Aut{\operatorname{Aut}}
\newtheorem{thm}{Theorem}[section]
\newtheorem{lem}[thm]{Lemma}

\theoremstyle{definition}

\newtheorem{prop}[thm]{Proposition}


\theoremstyle{rem}
\numberwithin{equation}{section}
\title[Module structure of the center ]{Module structure of the center of the universal central extension of a genus zero Krichever-Novikov algebra.}
\author{Ben Cox}
\address{Department of Mathematics \\
The University of Charleston \\
66 George Street  \\
Charleston SC 29424, USA}\email{coxbl@cofc.edu}
\urladdr{http://coxbl.people.cofc.edu/papers/preprints.html}

\usepackage{amssymb}
\begin{document}
 \begin{abstract}  We describe how the center of the universal central extension of the genus zero Krichever-Novikov current algebra  decomposes as a direct sum of irreducible modules for automorphism group of the coordinate ring of this algebra.  
\end{abstract}

\maketitle
\section{Introduction}

   Let $a_1,\dots, a_N$ be distinct complex numbers.  In previous work of the author together with  X. Guo, R. Lu, and K. Zhao, we described explicitly the automorphism groups of $\text{Der}\,(R)$ for the rings $R= \mathbb C[t,(t-a_1)^{-1},\dots, (t-a_N)^{-1}]$ where the possible groups that can appear are the cyclic groups $C_n$, the dihedral groups $D_n$, $S_4$, $A_4$ and $A_5$ (see \cite{MR3211093}).  This famous list of groups discovered by F. Klein appear naturally in the McKay correspondence and in the study of resolution of singularities.   
   
In the present work we describe in a series of propositions the irreducible subrepresentations that appear in the decomposition of $ \Omega_R^1/dR$ under the action of $\text{Aut}(R)$ for the various groups listed above.  The irreducible representations that do appear all have multiplicity one, but not all irreducible representations make their appearance.  The proofs use classical results of Schur and Frobenius.  These techniques can be found in nearly any book on the representation theory of finite groups such as Serre's book \cite{MR0450380} or Fulton and Harris \cite{MR1153249}.  We don't know of any conceptual nor geometric reason why not all irreducible representations appear and when they appear why they occur with multiplicity one.

From the work of S. Block, C. Kassel and J.L. Loday (see \cite{MR618298}, \cite{MR694130}, and \cite{MR772062}) we know if $\mathfrak g $ is a simple Lie algebra and $R$ is a commutative algebra, both defined over the complex numbers, then the universal central extension $\hat{{\mathfrak g}}$ of ${\mathfrak g}\otimes R$ is the vector space $\hat{{\mathfrak g}}:=\left({\mathfrak g}\otimes R\right)\oplus \Omega_R^1/dR$ where $\Omega_R^1/dR$ is the space of K\"ahler differentials modulo exact forms (see \cite{MR772062}).  The vector space $\hat{{\mathfrak g}}$ then becomes a Lie algebra by defining
$$
[x\otimes f,y\otimes g]:=[xy]\otimes fg+(x,y)\overline{fdg},\quad [x\otimes f,\omega]=0
$$
for $x,y\in\mathfrak g$, $f,g\in R$,  $\omega\in \Omega_R^1/dR$ and $(-,-)$ denotes the Killing form on $\mathfrak g$.  In the above $\overline{a}$ denotes the image of $a\in\Omega^1_R$ in the quotient space  $\Omega^1_R/dR$.    When $R$ is the ring of meromorphic functions on Riemann surface with a fixed finite number of poles, the algebra $\hat{\mathfrak g}$ is called a current Krichever-Novikov algebra and has been extensively studied (see for example the books \cite{MR3237422} and \cite{MR2985911} and their references). There are a number of natural questions that arise in studying such Lie algebras and one of them arises as follows.  

It is known that $l$-adic cohomology groups tend to be acted on by Galois groups, and the way in which these cohomology groups decompose can give interesting and important number theoretic information (see for example R. Taylor's review of Tate's conjecture \cite{MR2060030}).  Moreover it is an interesting and very difficult problem to describe the group $\text{Aut}(R)$ where $R$ is the space of meromorphic functions on a compact Riemann surface $X$ and to determine the module structure of its induced action on the module of holomorphic differentials $\mathcal H^1(X)$ (see \cite{MR1796706}).    Now if one realizes the fact that $\Omega_R^1/dR$ can be identified with the $H_2(\mathfrak{sl}(R),\mathbb C)$ (see \cite{MR618298}),  it is natural to ask how $\Omega_R^1/dR$ decomposes into a direct sum of irreducible modules under the action of the $\text{Aut}(R)$.  We answer this question when $R$ is the $N$-point algebra $R=\mathbb C[t,(t-a_1)^{-1},\dots, (t-a_N)^{-1}]$ with $a_1,\dots,a_N$  distinct complex numbers giving rise to the appropriate Kleinian groups $\text{Aut}(R)$.

   Let us give a bit of background where these algebras  appear in conformal field theory. The ring of functions  on the Riemann sphere regular everywhere except at a finite number of points appears naturally in Kazhdan and Luszig's explicit study of the tensor structure of modules for affine Lie algebras (see \cite{MR1186962} and \cite{MR1104840}).   M. Bremner gave these the name of an {\it $N$-point algebra} (see \cite{MR1261553}).   In the monograph  \cite[Ch. 12]{MR1849359} algebras of the form $\oplus _{i=1}^N\mathfrak g((t-x_i))\oplus\mathbb Cc$ appear in the description of the conformal blocks.  These contain the $n$-point algebras $ \mathfrak g\otimes \mathbb C[t,(t-a_1)^{-1},\dots, (t-a_N)^{-1}]\oplus\mathbb Cc$ modulo part of the center $\Omega_R/dR$.   M. Bremner explicitly described the universal central extension of such an algebra in \cite{MR1261553} where the center has basis $\overline{(t-a_1)^{-1}\,dt},\dots, \overline{(t-a_N)^{-1}\,dt}$.
 The current algebra $\mathfrak g\otimes R\oplus \Omega_R^1/dR $ and the derivation algebra $\text{Der}(R)$ are examples of Krichever-Novikov algebras for the genus zero Riemann sphere minus a finite number of points $R=\mathbb C[t,(t-a_1)^{-1},\dots, (t-a_N)^{-1}]$.  In Krichever and Novikov's original work they only dealt with a particular one dimensional central extension (see for example  \cite{MR925072}, \cite{MR902293},  \cite{MR998426}, \cite{MR3237422} and \cite{MR2985911}). 
  Also recently M. Schlichemaier has written a paper giving a different point of view of the universal central extension of the $N$-point Virasoro algebra (see \cite{Schlichenmaier2015}).

   We should also mention work of S. Lombardo,  A. V. Mikhailov, and J. A. Sanders dealing with the structure of genus zero multi-point Krichever-Novikov current al-
gebras (and their modules) fixed under the same type of symmetries as in our work with X. Guo, R. Lu,  and K. Zhao (see \cite{MR2166845}, \cite{MR2331220}, \cite{MR2718933}, and \cite{MR3254844}).  They do not describe how the center of the universal central extension decomposes into a direct sum of irreducible modules for the automorphism group of the coordinate ring.  

\section{The Universal Central Extension of the Current Algebra $\mathfrak g\otimes \mathcal R$.}
 Let $R$ be a commutative algebra defined over $\mathbb C$.
Consider the left $R$-module  $F=R\otimes R$ with left action given by $f( g\otimes h ) = f g\otimes h$ for $f,g,h\in R$ and let $K$  be the submodule generated by the elements $1\otimes fg  -f \otimes g -g\otimes f$.   Then $\Omega_R^1=F/K$ is the module of {\it K\"ahler differentials}.  The element $f\otimes g+K$ is traditionally denoted by $fdg$.  The canonical map $d:R\to \Omega_R^1$ is given by $df  = 1\otimes f  + K$.  The {\it exact differentials} are the elements of the subspace $dR$.  The coset  of $fdg$  modulo $dR$ is denoted by $\overline{fdg}$.  As C. Kassel showed the universal central extension of the current algebra $\mathfrak g\otimes R$ where $\mathfrak g$ is a simple finite dimensional Lie algebra defined over $\mathbb C$, is the vector space $\hat{\mathfrak g}=(\mathfrak g\otimes R)\oplus \Omega_R^1/dR$ with Lie bracket given by
$$
[x\otimes f,Y\otimes g]=[xy]\otimes fg+(x,y)\overline{fdg},  [x\otimes f,\omega]=0,  [\omega,\omega']=0,
$$
  where $x,y\in\mathfrak g$, and $\omega,\omega'\in \Omega_R^1/dR$ and $(x,y)$  denotes the Killing  form  on $\mathfrak g$.  
  
\begin{prop}[\cite{MR1261553}]\label{uce}  Let $a=(a_1,\dots, a_N)$ and $ R_a= \mathbb C[t,(t-a_1)^{-1},\dots, (t-a_N)^{-1}]$ be as above.  The set 
$$
\{\omega_1:=\overline{(t-a_1)^{-1} dt},\dots, \omega_N:=\overline{(t-a_N)^{-1}\,dt}\}
$$
 is a basis of $\Omega_{R_a}^1/d R_a$.
\end{prop}

\section{Recollections}  In \cite{MR3211093} one of the main results we proved together with Rencai Lu, Xiangqian Guo, and Kaiming Zhao was the following result.

\begin{thm}\label{ringisothm}
Suppose $\{a_1, a_2, \cdots, a_n\}$ and $\{a'_1, a'_2, \cdots,
a'_m\}$ are two sets of distinct complex numbers and $\phi: R_a\to
R_{a'}$ is an isomorphism of commutative algebras. Then $m=n$, $\phi$ is a fractional linear transformation, and
one of the following two cases holds
\begin{enumerate}
\item[(a).] There exists some constant $c\in\mathbb C$ such that
$a_i-a_1=c(a'_i-a'_1)$ and $\phi((t-a_i)^{k} )=c^{k}(t-a'_i)^{k} $
for all $k\in\mathbb Z$ and $i=1,2,\cdots ,n$ after reordering $a'_i$ if
necessary.
\item[(b).] There exists some constant $c\in\mathbb C$ such that
$(a_i-a_1)(a'_i-a'_1)=c$ for all $i\ne1$, and
$$
\phi((t-a_1)^k)=\frac{c^{k}}{(t-a'_1)^{k}},
$$
$$
\phi((t-a_i)^k)=\frac{(a_1-a_i)^k(t-a'_i)^k}{(t-a'_1)^{k}}, \,\,\forall\,\,i>1,
$$
for all $k\in\mathbb Z$ after reordering $a_i$ and $a'_i$ if necessary.
\end{enumerate}
Moreover, the maps given in (a) and (b) indeed define algebra
isomorphisms between $R_a$ and $R_{a'}$ under the corresponding
conditions.
\end{thm}

  We call isomorphisms defined in (a) and (b)
{\bf isomorphisms of the first kind} and {\bf isomorphisms of the
second kind} respectively.
 
 We also note the following theorems. 

\begin{thm}[\cite{MR1427489} Theorem 2.3.1] Any finite subgroup of the automorphism group $\text{Aut}(\hat{\mathbb C})$ of the Riemann sphere $\hat{\mathbb C}=\mathbb C\cup \{\infty\}$ is conjugate to a rotation group.
\end{thm}

\begin{thm}[\cite{MR1427489} Theorem 2.6.1]\label{Klein} If $G$ is a finite rotation group of the group of automorphisms of the Riemann sphere, then it is isomorphic to one of the following:
\begin{enumerate}[(a).]
\item  A cyclic group $C_n=\{s|s^n=1\}$ for $n\geq 1$.
\item  A dihedral group $D_n=\langle s,t\,|\, s^n=1=t^2,\enspace tst=s^{-1}\rangle$.
\item  A platonic rotation group $A_4,S_4$ or $A_5$.
\end{enumerate}
\end{thm}
These are the only groups that appear as automorphisms  groups of $R_a$ and all of them appear (see the appendix of \cite{MR3211093} and the lemmas below).

\section{Decompositions of the center under action of $\text{Aut}(R_a)$} 
Let $R$ be an associative commutative ring defined over a field $\mathbb F$.  We first observe that $\text{Aut}(R)$ acts on $\Omega_{R}/dR$ via the tensor product action on $R\otimes_{\mathbb F}R$ since this action preserves $K$ and $dR$. 

\subsection{Decomposition of $\Omega_{R_a}/dR_a$ under $C_n$.}
We give in this section examples of various $a$  that give the
automorphism groups of Klein listed in \thmref{Klein} as a series of
lemmas followed by propositions that describe how the centers decompose under the action of the respective automorphism groups.

\begin{lem}[\cite{MR3211093}, Appendix]\label{ringisolem}
Let $n\ge4$ and $a=(a_1, a_2, \cdots, a_n)$ where $ a_{k}=\zeta^{k}$
for the primitive $n$-th root of unity $\zeta=\exp(2\pi \imath /n)$.  Then
Aut(${\mathcal{V}}_a)\cong C_n$, the cyclic group of order $n$.
\end{lem}
 In the proof of this Lemma it was noted that
all automorphisms of $R_a$ where of the first kind
\begin{equation*}
\phi_k(t)=\zeta^kt,\,\, k=0,1,\cdots, n-1
\end{equation*}
with $\phi_k=\phi^k_1$.

Fix $n\geq 2$.  Let $\zeta_i=\exp(2\pi\imath /n)$, $\omega_i=\overline{(t-a_i)^{-1}dt}$, and $R_a=\mathbb C[t,(t-a_1)^{-1},\dots,(t-a_n)^{-1}]$. 

\begin{prop}   Let $a=(\zeta,\zeta^2,\dots,\zeta^{n-1},1)$.  The center $\Omega_{R_a}/dR_a$ is the direct sum of each irreducible representation of $C_n$ with multiplicity one.   More precisely 
\begin{equation}
\Omega_{R_a}/dR_a=\oplus_{k=0}^{n-1}\mathbb Cu_k
\end{equation}
where 
$$
u_k=\sum_{i=1}^{n}\zeta^{ki}\omega_i
$$
are eigenvectors for $C_n$ with eigenvalue $\zeta^{k}$ for the automorphism $\phi(z)=\zeta z$, $0\leq k\leq n-1$. 
In particular the subspaces $\mathbb Cu_n$ are distinct irreducible one dimensional $C_n$ submodules of $\Omega_{R_a}/dR_a$.
\end{prop}
\begin{proof}
The only automorphisms of $R_a$ are of the first kind and are powers of $\phi$ where
$$
\phi(z)=\zeta z.
$$
Hence
$$
\phi(\omega_i)
=\begin{cases} \omega_{i-1},&\quad \text{ for }2\leq i\leq n \\
 \omega_{n},&\quad \text{ for }i=1 .\end{cases}
$$
As a consequence
\begin{align*}
\phi(u_k)&=\sum_{i=1}^{n}\zeta^{ki}\phi(\omega_i)=\zeta^{k}\phi(\omega_1)+\sum_{i=2}^{n}\zeta^{ki}\phi(\omega_i) \\
&=\zeta^{k}\omega_n+\sum_{i=2}^{n}\zeta^{ki}\omega_{i-1} =\zeta^{k}\omega_n+\sum_{i=1}^{n-1}\zeta^{k(i+1)}\omega_i  \\
&=\zeta^{k}\left(\omega_n+\sum_{i=1}^{n-1}\zeta^{ki}\omega_i\right)=\zeta^{k}\left(\sum_{i=1}^n\zeta^{ki}\omega_i\right) \\  &=\zeta^{k}u_k.
\end{align*}
(See \cite{MR0450380} Section 5.1 for more information on the representation theory of $C_n$.)
\end{proof}

\subsection{Decomposition of $\Omega_{R_a}/dR_a$ under $S_4$ and $D_n$.}



\begin{lem}[\cite{MR3211093}, Appendix]
Let $n\ge3$, and $a=(0, a_1, a_2, \cdots, a_n)$ where $
a_{k}=\zeta^{k}$ for a primitive $n$-th root of unity $\zeta$.
\begin{enumerate}
\item[(a).]
 If $n\ne 4$ then  Aut(${\mathcal{V}}_a)\cong D_n$, the dihedral
group of order $2n$. \item[(b).] If $n= 4$ then
Aut(${\mathcal{V}}_a)\cong S_4$.\end{enumerate}
\end{lem}
In the proof of this Lemma it was noted that the only automorphisms of $R_a$ are products of the automorphims $\phi$ and $\psi$ where 
\begin{equation}
\phi(t)=\zeta t,\quad\text{and}\quad \psi(t)=\zeta/t.
\end{equation}

Fix $n\geq 2$.  Let $\zeta=\exp(2\pi\imath /n)$, $a_i=\zeta^i$, $\omega_0=\overline{t^{-1}\,dt}$, $\omega_i=\overline{(t-a_i)^{-1}dt}$,  and $R_a=\mathbb C[t,(t-a_1)^{-1},\dots,(t-a_n)^{-1}]$.

\subsubsection{The decomposition for $D_n$}

\begin{prop}   Let $a=(0,\zeta,\zeta^2,\dots,\zeta^{n-1},1)$.  The center $\Omega_{R_a}/dR_a$ is the direct sum of all of the irreducible two dimensional representation of $D_n$ with multiplicity one together with at most three other distinct irreducible representations.  More precisely 
\begin{enumerate}[i).]
  \item If $n=2m\neq 4$ is even, then 
$$
\Omega_{R_a}/dR_a=\mathbb C\omega_0\oplus\left( \oplus_{k=1}^{m-1}U_k\right)\oplus , \mathbb Cu_{m}\oplus \mathbb Cu_n
$$
where 
$$
U_k=\mathbb Cu_k\oplus \mathbb Cu_{n-k},\enspace 1\leq k\leq m-1,\quad u_k=\sum_{i=1}^{n}\zeta^{ki}\omega_i,\quad 1\leq k\leq n.
$$
  The subspaces $\mathbb C\omega_0$, $ \mathbb Cu_{m}$, $\mathbb Cu_n$ and the $U_k$ are irreducible subrepresentations for $D_{n}=\langle \phi,\psi\rangle$ for the automorphism $\phi(t)=\zeta t$ and $\psi(t)=\zeta/t$.   The irreducible modules $U_k$, $1\leq k\leq m-1$ exhaust all of the irreducible two dimensional modules for $D_n$. 
  \item If $n=2m+1$ is odd, then 
$$
\Omega_{R_a}/dR_a=\mathbb C\omega_0\oplus\left( \oplus_{k=1}^{m}U_k\right)
\oplus \mathbb Cu_n
$$
where   
$$
U_k=\mathbb Cu_k\oplus \mathbb Cu_{n-k},\enspace 1\leq k\leq m-1\quad u_k=\sum_{i=1}^{n}\zeta^{ki}\omega_i,\quad 1\leq k\leq n.
$$
 The subspace $\mathbb C\omega_0$, $\mathbb Cu_n$ and the $U_k$ are irreducible subrepresentations for $D_{n}$.  The irreducible modules that appear in the above decomposition exhaust all of the irreducible modules for $D_n$ with multiplicity one. 
\end{enumerate}
\end{prop}
\begin{proof}
Recall only automorphisms of $R_a$ are products of the automorphims $\phi$ and $\psi$ where 
$$
\phi(t)=\zeta t,\quad\text{and}\quad \psi(t)=\zeta/t.
$$
 Now we calculate 
$$
\phi(\omega_0)=\overline{\phi(t)^{-1}\,d\phi(t)}=\overline{\zeta^{-1}t^{-1}\zeta\,d t}=\omega_0
$$
$$
\psi(\omega_0)=\overline{\psi(t)^{-1}\,d\psi(t)}=\overline{\zeta^{-1}t\zeta\,d t^{-1}}=-\overline{t^{-1}\,d t}=-\omega_0.
$$
As a consequence the subspace $\mathbb C\omega_0$ is a one dimensional irreducible representation of $D_{n}$ occurring as a direct summand of $\Omega_{R_a}/dR_a$. 

As above $\phi(u_i)=\zeta^iu_i$.
On the other hand 
\begin{align*}
\psi(\omega_i)&=\overline{\psi(t-\zeta^i)^{-1}\,d\psi(z)}=\overline{((\zeta/t)-\zeta^i)^{-1}\,\zeta d t^{-1}}\\
&=-\zeta^{1-i} \overline{(\zeta^{1-i}-t)^{-1}t^{-1}\,d t}
\\
&=\zeta^{1-i}\left(-\zeta^{i-1}\overline{t^{-1}\,d t}+\zeta^{i-1}\overline{(t-\zeta^{1-i})^{-1}\,d t}\right) \\
&= - \overline{t^{-1}\,d t}+ \overline{(t-\zeta^{1-i})^{-1}\,d t}  \\
&= - \omega_0+ \omega_{n+1-i},
\end{align*}
and so 
\begin{align*}
\psi(u_k)&=\sum_{i=1}^{n}\zeta^{ki}\psi(\omega_i)=-\left(\sum_{i=1}^{n}\zeta^{ki}\right)\omega_0+\sum_{i=1}^{n}\zeta^{ki}\omega_{n+1-i} \\
&=\sum_{m=1}^{n}\zeta^{k(n-m+1)}\omega_{m}=\zeta^k\sum_{m=1}^{n}\zeta^{(n-k)m}\omega_{m} \\
&=\zeta^ku_{n-k}.
\end{align*}
For $k=n$ we have $\psi(u_n)=u_{0}=u_n$.   Similarly $\phi(u_n)=u_n$ and as a consequence $\mathbb Cu_n$ is the trivial irreducible representation of $D_{n}$.

Suppose $n$ is even with $n=2m$, $m\in\mathbb Z$.  Then $\psi(u_m)=\zeta^m u_{m}=-u_{m}$.   Similarly $\phi(u_m)=\zeta^mu_m=-u_m$.  This means that $\mathbb Cu_m$ is the one dimensional alternating representation. 

Let $n=2m$ be still even. Now suppose $1\leq k\leq m-1$.   If we set $w_k=u_k$ for $1\leq k\leq m-1$  and $w_{n-k}=\zeta^ku_{n-k}$ for $m\leq k\leq n-1$, then the matrices for $\phi$ and $\psi$ for the basis $\{w_k,w_{n-k}\}$, $1\leq k\leq m-1$ are 
\begin{equation}\label{twodimirrep}
\phi=\begin{pmatrix} e^{2\pi\imath k/n} & 0  \\ 0 & e^{-2\pi\imath k/n} \end{pmatrix},\quad \psi=\begin{pmatrix}0 & 1 \\ 1 & 0  \end{pmatrix}.
\end{equation}

Suppose now $n=2m+1$ is odd. If we set $w_k=u_k$ for $1\leq k\leq m$  and $w_{n-k}=\zeta^ku_{n-k}$ for $m+1\leq k\leq n-1$, then the matrices for $\phi$ and $\psi$ for the basis $\{w_k,w_{n-k}\}$ are as in \eqnref{twodimirrep}.  

(See \cite{MR0450380} Section 5.3 for more information on the representation theory of $D_n$.)
\end{proof}

For $D_n$ we are left with considering the case of $D_4$.   
\begin{lem}[\cite{MR3211093}, Example 3. Case 2] Let $a=(0,1,-1)$.  Then we have the automorphism of the first kind
$\tau (t)=-t$ and an automorphism of the second kind 
$$\sigma(t)=\frac {t+1}{t-1}$$
 Then $\Aut(\mathcal V_a)=H\cup H\tau\cong D_4=\langle x,b\,|\, x^4=b^2=1,\enspace bxb=x^{-1}\rangle$ where for example $x=\sigma\tau$ and $b=\tau$, in this case.

\end{lem}

 \begin{prop}
 Again suppose $a=(0,1,-1)$.  Let $\omega_0=\overline{t^{-1}\,dt}$, $\omega_1=\overline{(t-1)^{-1}\,dt}$ and $\omega_2=\overline{(t+1)^{-1}\,dt}$.   Then 
$ \Omega_{R_a}/dR_a$,
 is the direct sum of a one dimensional representation $\mathbb Cu_1$ where $u_1=\omega_0-\omega_1-\omega_2$ and the unique irreducible representation of $D_4$ of dimension $2$ which we denote by $U_2=\mathbb Cu_2+\mathbb Cu_3$ where $u_2=\omega_0$, $u_3=\omega_1-\omega_2$. 
\end{prop}

\begin{proof}

There are four irreducible one dimensional representations which  we denote by $\psi_i$, $i=1,2,3,4$ and one irreducible two dimensional representation which we denote by $\chi_2$. 
The character table of $D_4$ is 

\begin{table}[htp]
\caption{Character table of $D_4$.}
\begin{center}
\begin{tabular}{|c|c|c|c|}
\hline
& $x^k$ & $\tau x^k$     \\
\hline
$\psi_{1}$ & 1 & 1    \\
\hline
$\psi_{2}$& 1 & -1    \\
\hline
$\psi_{3}$& $(-1)^k$ & $(-1)^k$    \\
\hline
$\psi_{4}$& $(-1)^k$ & $(-1)^{k+1}$  \\
\hline
$\chi_{2}$& $2(\delta_{k,0}-\delta_{k,2})$ & 0    \\
\hline
\end{tabular}
\end{center}
\label{default}
\end{table}%
for $0\leq k\leq 3$.   (See \cite{MR0450380} Section 5.3.) Here $\delta_{k,l}$ is the Kronecker delta function.  

The one irreducible representation $\rho$ of $D_4$ of dimension 2  defined on $\mathbb C^2$ is given by 
$$
\rho(x^k)=\begin{pmatrix} e^{k\imath \pi/2} & 0 \\ 
0& e^{-k\imath \pi/2} \end{pmatrix} ,\quad \rho(\tau x^k)=\begin{pmatrix} 0 &e^{-k\imath \pi/2} \\ 
e^{k\imath \pi/2} & 0 \end{pmatrix}.
$$
There are five conjugacy classes of $D_4$ and they are $\{I\}$, $\{x^2\}$, $\{\tau,x^2\tau\}$, $\{x\tau,x^3\tau\}$, and $\{x,x^3\}$. 

On the representation $\Omega_{R_a}/dR_a$ we have 
\begin{align*}
\tau(\omega_0)&=\overline{\tau(t^{-1})\,d\tau(t)}=\overline{t^{-1}\,dt}=\omega_0,  \\
\tau(\omega_1)&=\overline{\tau((t-1)^{-1})\,d\tau(t)}=\overline{(t+1)^{-1}\,dt}=\omega_2,  \\  
\tau(\omega_2)&=\overline{\tau((t+1)^{-1})\,d\tau(t)}=\overline{(t-1)^{-1}\,dt}=\omega_1, \\  \\
x\tau (\omega_0)&=\sigma(\omega_0)=\overline{\sigma(t^{-1})\,d\sigma(t)}=\overline{\frac {t-1}{t+1}\,d\left(\frac {t+1}{t-1}\right)}  \\
&=-2\overline{\frac {1}{(t-1)(t+1)}\,dt}=\overline{ (t+1)^{-1}\,dt}  -\overline{ (t-1)^{-1}\,dt} =\omega_2-\omega_1, \\
x\tau (\omega_1)&=\sigma(\omega_1)=\overline{\sigma((t-1)^{-1})\,d\sigma(t)}=\overline{\left(\frac{t+1}{t-1}-1\right)^{-1}\,d\left(\frac {t+1}{t-1}\right)}  \\
&=- \overline{ (t-1)^{-1}\,dt}=-\omega_1 \\
x\tau (\omega_2)&=\sigma(\omega_2)=\overline{\sigma((t+1)^{-1})\,d\sigma(t)}=\overline{\left(\frac{t+1}{t-1}+1\right)^{-1}\,d\left(\frac {t+1}{t-1}\right)}  \\
&= \overline{ t^{-1}\,dt}- \overline{ (t-1)^{-1}\,dt}=\omega_0-\omega_1.
\end{align*}
Thus 
\begin{align*}
x(\omega_0)&=\sigma\tau (\omega_0)=\sigma(\omega_0)=\omega_2-\omega_1, \\
x(\omega_1)&=\sigma\tau (\omega_1)=\sigma (\omega_2)=\omega_0-\omega_1,  \\
x(\omega_2)&=\sigma \tau (\omega_2)= \sigma ( \omega_1)=-\omega_1,\\ \\
x^2(\omega_0)&=x(\omega_2-\omega_1)=-\omega_1-(\omega_0-\omega_1)=-\omega_0, \\
x^2(\omega_1)&=x  (\omega_0-\omega_1)=\omega_2-\omega_1-(\omega_0-\omega_1)=\omega_2-\omega_0,  \\
x^2(\omega_2)&=x   (-\omega_1)= \omega_1-\omega_0.
\end{align*}
As a consequence we have $\chi_{\Omega_{R_a}/dR_a}(\tau)=1$,  $\chi_{\Omega_{R_a}/dR_a}(x\tau )=-1 =\chi_{\Omega_{R_a}/dR_a}(x  )=\chi_{\Omega_{R_a}/dR_a}(x^2  )$  .
We know that there exists a unique set of nonnegative integers $n_1,n_2,n_3,n_4,m$ such that 
\begin{equation}
\chi_{\Omega_{R_a}/dR_a}=n_1\psi_1+n_2\psi_2+n_3\psi_3+n_4\psi_4+m\chi_2.
\end{equation}
Then
\begin{align*}
3&=n_1 +n_2 +n_3 +n_4 +2m  \\
1&=n_1\psi_1(\tau)+n_2\psi_2(\tau)+n_3\psi_3(\tau)+n_4\psi_4(\tau)+m\chi_2(\tau) \\
&=n_1 -n_2 +n_3-n_4   \\
-1&=n_1\psi_1(x \tau )+n_2\psi_2(x\tau )+n_3\psi_3(x \tau   )+n_4\psi_4(x \tau ) +m\chi_2(x \tau )\\
&=n_1 -n_2 -n_3 +n_4   \hskip 50pt \text{ as } x \tau =\tau x^3\\ 
-1&=n_1\psi_1(x )+n_2\psi_2(x )+n_3\psi_3(x  )+n_4\psi_4(x) +m\chi_2(x )\\
&=n_1 +n_2 -n_3 -n_4 \\ 
-1&=n_1\psi_1(x^2 )+n_2\psi_2(x^2 )+n_3\psi_3(x^2 )+n_4\psi_4(x^2) +m\chi_2(x^2 )\\
&=n_1 +n_2 +n_3 +n_4-2m.
\end{align*}
If $m=0$ then the last equation is inconsistent. Moreover as $\dim \Omega_{R_a}/dR_a=3$,  $m\leq 1$ so we must have $m=1$. 
Then the only solution to the above equations is $n_1=n_2=n_4=0$ and $n_3=1$.  Thus $\chi_{\Omega_{R_a}/dR_a}=\psi_3+\chi_2$.  

To find the irreducible subspaces we need the projection formula with respect to the basis $\omega_0,\omega_1,\omega_2$ for the representation $\rho_{D_4}$ of $D_4$ on $\Omega_{R_a}/dR_a$:
\begin{align*}
\pi_{\chi_{2}}&=\frac{1}{|D_4|}\sum_{g\in S_4}\chi_2(g)\rho_{D_4}(g)  \\
&=\frac{1}{8}\left(2 \begin{pmatrix} 1 & 0 & 0 \\ 0 & 1 & 0 \\ 0 & 0 & 1\end{pmatrix} -2\begin{pmatrix} -1 &-1 & -1 \\ 0 & 0 & 1 \\ 0 & 1 & 0\end{pmatrix}  \right)  \\
&=\frac{1}{4}  \begin{pmatrix} 2 & 1 &  1 \\ 0 & 1 & -1 \\ 0 & - 1 & 1\end{pmatrix}.
\end{align*}
Thus a basis for $U_2$ is $u_2=\omega_0$, $u_3=\omega_1-\omega_2$.

The  projection formula for $\psi_3$ with respect to the basis $\omega_0,\omega_1,\omega_2$ is given by
\begin{align*}
\pi_{\psi_{3}}&=\frac{1}{|D_4|}\sum_{g\in S_4}\psi_3(g)\rho_{D_4}(g)  \\
&=\frac{1}{8}\Big( \begin{pmatrix} 1 & 0 & 0 \\ 0 & 1 & 0 \\ 0 & 0 & 1\end{pmatrix}  +\begin{pmatrix} -1 &-1 & -1 \\ 0 & 0 & 1 \\ 0 & 1 & 0\end{pmatrix} +\begin{pmatrix} 1 & 0& 0\\ 0 & 0 & 1 \\ 0 & 1 & 0\end{pmatrix}+\begin{pmatrix}-1 & -1 & -1 \\
 0 & 1 & 0 \\
 0 & 0 & 1 \\
\end{pmatrix} \\
&\quad  -\begin{pmatrix} 0 &0  &1  \\ -1 & -1 & -1 \\ 1 & 0 & 0\end{pmatrix} -\begin{pmatrix} 0 & 0 & 1 \\
 1 & 0 & 0 \\
 -1 & -1 & -1 \end{pmatrix}    -\begin{pmatrix}0 & 1 & 0 \\  -1 &-1 & -1 \\ 1 & 0 & 0\end{pmatrix}  -\begin{pmatrix} 1 & 0 & 0 \\
 1 & 0 & 0 \\
 -1 & -1 & -1 \end{pmatrix}  \Big)  \\
&=\frac{1}{8}\begin{pmatrix}
 0 & -4 & -4 \\
 0 & 4 & 4 \\
 0 & 4 & 4 \\
\end{pmatrix}.
\end{align*}
Thus $\omega_0-\omega_1-\omega_2$ is a basis of the one dimensional irreducible subrepresentation whose character is $\psi_3$.

\end{proof}

\subsubsection{The decomposition of $\Omega_{R_a}/dR_a$ for $S_4$}
\begin{lem}[\cite{MR0080930}, \cite{MR1427489} and \cite{MR3211093} Appendix]
The automorphism group of $\mathcal V_{(0,1, \imath,-1, -\imath)}$ is $S_4$.
\end{lem}

Here the automorphism group of $\mathcal V_{(0,1, \imath,-1, -\imath)}$ is generated by an automorphism of the first kind
$$
\phi(t)=\imath t
$$
and also an automorphism of the second kind given by
$$
\psi(t)=\frac{t+\imath}{t-\imath}.
$$

\begin{prop}
Set $a={(0,1, \imath,-1, -\imath)}$.  For the automorphism group of $\mathcal V_a$, $S_4$, we have 
\begin{equation}
\Omega_{R_a}/dR_a=U_\theta \oplus U_{\rho\epsilon}
\end{equation}
where 
\begin{align*}
U_\theta&=\mathbb C(-\omega_0+\omega_1+\omega_3)+\mathbb C(-\omega_0+\omega_2+\omega_4) \\
 U_{\rho\epsilon}&=\mathbb C\omega_0+ \mathbb C(-\omega_1+\omega_3)+ \mathbb C(-\omega_2+\omega_4)
 \end{align*}
 are the irreducible $S_4$-subrepresentations of $\Omega_{R_a}/dR_a$.   
\end{prop}
\begin{proof}
First observe
\begin{align*}
\phi(\omega_0)&=-\imath^2 \overline{t^{-1}\,d t}=\omega_0, \\ 
\phi(\omega_1)
&= \overline{ ( t+\imath)^{-1}\,d t} = \omega_4, \\ 
\phi(\omega_2)
&=\overline{ ( t-1)^{-1}\,dt} = \omega_1, \\ 
\phi(\omega_3)
&=\overline{ ( t-\imath)^{-1}\,dt} = \omega_2, \\ 
\phi(\omega_4)
&= \overline{ ( t+1)^{-1}\,d t} = \omega_3,
\end{align*}
and
\begin{align*}
\psi(\omega_0)
&=-\omega_2+\omega_4,\\ 
\psi(\omega_1)
& =-\omega_2 , \\
\psi(\omega_2)
&=-\omega_2+\omega_1, \\
\psi(\omega_3)
&=-\omega_2+\omega_0, \\
\psi(\omega_4),
&= -\omega_2 +\omega_3.
\end{align*}
Thus the linear transformations $\phi$ and $\psi$ have with respect to this basis the matrices 
\begin{equation}
\phi=\begin{pmatrix}
 1 & 0 & 0 & 0 & 0 \\
 0 & 0 & 1 & 0 & 0 \\
 0 & 0 & 0 & 1 & 0 \\
 0 & 0 & 0 & 0 & 1 \\
 0 & 1 & 0 & 0 & 0 \\
\end{pmatrix},\hskip 50pt
\psi=\begin{pmatrix}
 0 & 0 & 0 & 1 & 0 \\
 0 & 0 & 1 & 0 & 0 \\
 -1 & -1 & -1 & -1 & -1 \\
 0 & 0 & 0 & 0 & 1 \\
 1 & 0 & 0 & 0 & 0 \\
\end{pmatrix}.
\end{equation}
We let $\chi_{\Omega_R/dR}=a\chi_0+b\chi_\epsilon+c\chi_\theta+d\chi_\rho+e\chi_{\epsilon\rho}$ where $\chi_0$ is the character of the trivial representation, $\chi_\epsilon$ is the character of the alternating representation, $\chi_\theta$ is the character of the unique two dimensional irreducible representation $\chi_\rho$, $\chi_\rho$ is the character of the natural representation on $\mathbb C^3$ and $\chi_{\epsilon\rho}$ is the character of the tensor product of the alternating representation and the natural representation. 

We then have the following values of the trace of the conjugacy classes
\begin{table}[htp]
\caption{Character of $S_4$ on $\Omega_R/dR$}
\begin{center}
\begin{tabular}{|c|c|c|c|c|c|}
\hline
&I & $\psi\phi$ & $\phi^2$ & $\psi$ & $\phi$  \\
\hline
$\chi_{\Omega_R/dR}$& 5 & -1 & 1 &-1 & 1  \\
\hline
$I$& 1 & 1 & 1 &1 & 1  \\
\hline
$\chi_{\epsilon}$& 1 & -1& 1 &1 &- 1  \\
\hline
$\chi_{\theta}$& 2& 0 & 2 & -1 & 0 \\
\hline
$\chi_{\rho}$& 3& 1& -1 & 0 &- 1  \\
\hline
$\chi_{\epsilon\rho}$& 3& -1& -1 & 0 & 1  \\
\hline
\end{tabular}
\end{center}
\label{default}
\end{table}%

Note $\phi$ has order 4 while $\psi $ has order 3.  The number of elements in the conjugacy class of $\psi\phi$ is 6:  $\psi\phi$, $\phi^3\psi\phi^2$, $\phi^2\psi\phi^3$, $\phi\psi$, $\psi^2\phi\psi^2$, $\phi\psi^2\phi\psi^2\phi^3$. 
The number of elements in the conjugacy class of $\phi^2$ is 3: $\phi^2$, $\psi\phi^2\psi^2$, $\psi^2\phi^2\psi $. 
The number of elements in the conjugacy class of $\psi$ is 8:  $\psi$, $\phi\psi\phi^3$,  $\phi^2\psi\phi^2$, $\phi^3\psi\phi$,    $\psi\phi^2\psi\phi^2\psi^2$,   $\psi^2\phi^2\psi\phi^2\psi$, $\phi\psi^2\phi^2\psi\phi^2\psi\phi^3$,    $\phi^3\psi^2\phi^2\psi\phi^2\psi\phi$.    The number of elements in the conjugacy class of $\phi$ is 6: $\phi$, $\psi\phi\psi^2$, $\psi^2\phi\psi$, $\phi^2\psi\phi\psi^2\phi^2$, $\phi^2\psi^2\phi\psi\phi^2$,
$\psi\phi^2\psi^2\phi\psi\phi^2\psi^2$.  

Solving 
\begin{align*}
5&=a+b+2c+3d+3e \\
-1&=a-b+d-e \\
1&=a+b+2c-d-e \\
-1&=a+b-c \\
1&=a-b-d+e
\end{align*}
we get $a=b=d=0$ and $c=1=e$.

Now we need to find a basis for these two irreducible components for $\chi_\theta$, $\chi_{\epsilon\rho}$.  Let  $\rho:S_4\to \text{GL}(\Omega_{R_a}/dR_a)$ is the given induced representation.  To find the irreducible component for $\chi_\theta$, we need the projection formula:
\begin{align*}
\pi_{\chi_{\theta}}&=\frac{1}{|S_4|}\sum_{g\in S_4}\chi_\theta(g)\rho(g)  \\
&=\frac{1}{24}\sum_{g\in S_4}\chi_\theta(g)\rho(g)  \\
&=\frac{1}{24}\begin{pmatrix}
  0 & -6 & -6 & -6 & -6 \\
 0 & 6 & 0 & 6 & 0 \\
 0 & 0 & 6 & 0 & 6 \\
 0 & 6 & 0 & 6 & 0 \\
 0 & 0 & 6 & 0 & 6 
\end{pmatrix}=\frac{1}{4}\begin{pmatrix}
  0 & -1 & -1 & -1 & -1 \\
 0 & 1 & 0 & 1 & 0 \\
 0 & 0 & 1 & 0 & 1 \\
 0 & 1 & 0 & 1 & 0 \\
 0 & 0 & 1 & 0 & 1 
\end{pmatrix}.
\end{align*}
Thus the subspace spanned $U_\theta$ by  $-\omega_0+\omega_1+\omega_3$ and $-\omega_0+\omega_2+\omega_4$ is an irreducible subrepresentation of $\Omega_{R_a}/dR_a$.

For the irreducible component of type $\chi_{\epsilon\theta}$ we have 
\begin{align*}
\pi_{\chi_{\epsilon\theta}}&=\frac{1}{|S_4|}\sum_{g\in S_4}\chi_{\epsilon\theta}(g)\rho(g)  \\
&=\frac{1}{24}\sum_{g\in S_4}\chi_{\epsilon\theta}(g)\rho(g)  \\
&=\frac{1}{6}\begin{pmatrix}
 2 & 1 & 1 & 1 & 1 \\
 0 & 1 & 0 & -1 & 0 \\
 0 & 0 & 1 & 0 & -1 \\
 0 & -1 & 0 & 1 & 0 \\
 0 & 0 & -1 & 0 & 1 \\
 \end{pmatrix}
\end{align*}
Thus a basis for the three dimensional irreducible subrepresentation with character $\chi_{\epsilon\rho}$ is $\omega_0,-\omega_1+\omega_3, -\omega_2+\omega_4$. 
\end{proof}

(In the Propositions  below the author will not be as detailed about the calculations, but they all are of a similar level of difficulty and if the reader is so interested in the details, please feel free to email him - the author will send you the tex file with nearly all calculations written out explicitly. )

\subsection{Decomposition of $\Omega_{R_a}/dR_a$ under $A_4$.}
\begin{lem}[\cite{MR3211093}, Example 3. Case 3] Let $a=(0,1,x)$ where $\displaystyle{x=\frac{1+\imath \sqrt{3}}{2}}$.  Then $\text{Aut}(\mathcal V_a)\cong A_4$, the alternating group on $4$ letters.   Here 
$$
\text{Aut}(\mathcal V_a)=H\cup H\tau_2\cup H\tau_2^2
$$
where
$H=\{1,\sigma_1,\sigma_2,\sigma_3\}$ and the automorphisms are defined by
\begin{align*}
\sigma_1(t)=\frac{x}{t},\enspace \sigma_2(t)=\frac{t-x}{t-1},\enspace 
\sigma_3(t)=\frac{x(t-1)}{t-x},\enspace \tau_2(t)=-x(t-1).
\end{align*}
\end{lem}
\begin{prop}
 Again suppose $a=(0,1,x)$ where $\displaystyle{x=\frac{1+\imath \sqrt{3}}{2}}$.  Let $\omega_0=\overline{t^{-1}\,dt}$, $\omega_1=\overline{(t-1)^{-1}\,dt}$ and $\omega_2=\overline{(t-x)^{-1}\,dt}$.   Then 
$ \Omega_{R_a}/dR_a$,
 is the unique irreducible representation of $A_4$ of dimension $3$. 
\end{prop}
\begin{proof}
Set $\zeta=\exp(2\pi\imath/3)$. Let $\chi_{3}$ be the character of the unique irreducible $3$-dimensional representation of $A_4$ and $\chi_1$ and $\chi_2$ the characters of the two unique non-trivial representations of $A_4$. The character table of $A_4$ is
\begin{table}[htp]
\caption{Character table of $A_4$.}
\begin{center}
\begin{tabular}{|c|c|c|c|c|}
\hline
&I & $\sigma_1$ & $\tau_2$ & $\tau_2^2$   \\
\hline
$I$& 1 & 1 & 1 &1   \\
\hline
$\chi_{1}$& 1 & 1& $\zeta$ &$\zeta^2$  \\
\hline
$\chi_{2}$& 1& 1 &$\zeta^2$& $\zeta$  \\
\hline
$\chi_{3}$& 3& -1& 0 & 0 \\
\hline
\end{tabular}
\end{center}
\label{default}
\end{table}%

Now one can calculate 
\begin{align*}
\sigma_1(\omega_0)
&=\overline{\sigma_1(t^{-1})\,d\sigma_1(t)}=\overline{t/x\,d(x/t)}=-\overline{t^{-1}\,dt}=-\omega_0,\\
\sigma_1(\omega_1)
&=-\omega_0+\omega_2 , \\
\sigma_1(\omega_2)
&=-\omega_0+\omega_1.
\end{align*}
Thus $\chi_{\Omega_{R_a}/dR_a}(\sigma_1)=-1$.   

Moreover 
\begin{align*}
\tau_2(\omega_0)
&=\overline{\tau_2(t^{-1})\,d\tau_2(t)}=\overline{(-x(t-1))^{-1}\,d(-x(t-1))}=\overline{(t-1)^{-1}\,dt}=\omega_1,\\ 
\tau_2(\omega_1)
&=\omega_2, \\
\tau_2(\omega_2)
&=\omega_0.
\end{align*}
and $\tau_2^2(\omega_0)=\omega_2$, $\tau_2^2(\omega_1)=\omega_0$, $\tau_2^2(\omega_2)=\omega_1$.  
As a consequence $\chi_{\Omega_{R_a}/dR_a}(\tau_2)=0=\chi_{\Omega_{R_a}/dR_a}(\tau_2^2)$.  Thus the character of $\Omega_{R_a}/dR_a$ is the same as the three dimensional irreducible representation for $A_4$ and thus it must be isomorphic to it.
\end{proof}

%
%
%

\subsection{Decomposition of $\Omega_{R_a}/dR_a$ under $A_5$.}
\begin{lem}[\cite{MR0080930}, \cite{MR1863996} and \cite{MR1427489}, 
\cite{MR3211093} Appendix]  Let $\zeta=\exp (2\pi \imath  /5)$ and
$$
a_0=0, \enspace a_i= \zeta^{i-1}(\zeta+\zeta^4),\enspace 1\leq i\leq
5, \quad a_i=\zeta^{i-6}(\zeta^2+\zeta^3),\enspace 6\leq i\leq 10.
$$
The automorphism group of $\mathcal V_{(0,a_1,\dots, a_{10})}$ is
$A_5$ where  the automorphisms are
\begin{gather*}
t\mapsto \zeta^jt,\quad t\mapsto -\frac{1}{\zeta^j t}, \\
t\mapsto \zeta^j\frac{-(\zeta-\zeta^4)\zeta^lt +(\zeta^2-\zeta^3)}{(\zeta^2-\zeta^3)\zeta^lt+(\zeta-\zeta^4)}  \\
t\mapsto \zeta^j\frac{(\zeta^2-\zeta^3)\zeta^lt+(\zeta-\zeta^4)}
{(\zeta-\zeta^4)\zeta^lt -(\zeta^2-\zeta^3)},
\end{gather*}
for $j,l=0,\dots 4$.
\end{lem}
As usual set $\omega_0=\overline{t^{-1}\,dt}$ and $\omega_i=\overline{(t-a_i)^{-1}\,dt}$ for $1\leq i\leq 10$.
\begin{prop}   The $Aut(R_a)$ module $\Omega_{R_a}/dR_a$ decomposes into a direct sum of irreducible submodules as 
\begin{align*}
\Omega_{R_a}/dR_a=U_3\oplus U_3'\oplus U_5
\end{align*}
where $U_3$ and $U_3'$ are the two distinct irreducible representations of $A_5$ of dimension 3 and $U_5$ is the unique irreducible representation of dimension 5.  A basis of $U_5$ is $\{-\omega_0+\omega_i+\omega_{i+5}\,|\,1\leq i\leq 5\}$ and bases of $U_3$ and $U_3'$ are respectively
\begin{align*}
&\sqrt{5}\omega_0+\sum_{i=1}^5(\omega_i-\omega_{i+5}), \\
 &\xi \omega_0+\xi \omega_1+\omega_2+\omega_5-\xi\omega_6-\omega_7-\omega_{10}, \\
 &\xi \omega_0+ \omega_1+\xi\omega_2+\omega_3-\omega_6-\xi\omega_7-\omega_8,
\end{align*}
and
\begin{align*}
&\sqrt{5}\omega_0-\sum_{i=1}^5(\omega_i-\omega_{i+5}), \\
 &-\bar\xi \omega_0-\bar\xi \omega_1-\omega_2-\omega_5+\bar\xi\omega_6+\omega_7+\omega_{10}, \\
 &-\bar\xi \omega_0- \omega_1-\bar\xi\omega_2-\omega_3+\omega_6+\bar\xi\omega_7+\omega_8,
\end{align*}
where $\displaystyle{\xi=\frac{1+\sqrt{5}}{2}}$ and $\displaystyle{\bar \xi=\frac{1-\sqrt{5}}{2}}$. 
\end{prop}
\begin{proof}

Let $\phi$ be the automorphism given by $\phi(t)=\zeta t$. Then a straightforward calculation yields
\[\begin{alignedat}{5}
\phi(\omega_0)
&=\omega_0, \quad 
\phi(\omega_1) 
&=\omega_5,\quad 
\phi(\omega_2)
&=\omega_1, \quad 
\phi(\omega_3)
&=\omega_2 ,\quad 
\phi(\omega_4)
&=\omega_3 ,\quad
\phi(\omega_5)
&=\omega_4,  \\
\phi(\omega_6)
&=\omega_{10} ,\quad 
\phi(\omega_7)
&=\omega_6, \quad
\phi(\omega_8)
&=\omega_7 ,\quad 
\phi(\omega_9)
&=\omega_8 ,\quad  
\phi(\omega_{10})
&=\omega_9.
\end{alignedat}
\]
Thus the $\chi_{\Omega_{R_a}/dR_a}(\phi^l)=1$ for $1\leq l\leq 4$. 

Let $\psi$ be the automorphism given by $\displaystyle{\psi(t)=-\frac{1}{\zeta t}}$.   In the calculations below note that $(\zeta+\zeta^4)(\zeta^2+\zeta^3)=-1$.  
Then we have 
\[ 
\begin{alignedat}{5} 
&\psi(\omega_1)
&=&-\omega_0+\omega_{10}\quad   
&\psi(\omega_2)
&=&-\omega_0+\omega_9,\quad  
&\psi(\omega_3)
&=&-\omega_0+\omega_8 \\   
&\psi(\omega_4)
&=&-\omega_0+\omega_7,\quad    
&\psi(\omega_5)
&=&-\omega_0+\omega_6 ,\quad   
&\psi(\omega_6)
&=&-\omega_0+\omega_{5} \\ 
&\psi(\omega_7)
&=&-\omega_0+\omega_{4},\quad &\psi(\omega_8)
&=&-\omega_0+\omega_{3},\quad
&\psi(\omega_9)
&=&-\omega_0+\omega_{2} \\ 
&\psi(\omega_{10})
&=&-\omega_0+\omega_{1} ,\quad 
&\psi(\omega_0)
&=&-\omega_0 &
\end{alignedat}
\]

Thus $\chi_{\Omega_{R_a}/dR_a}(\psi)=-1$ and $\psi^2=1$.

Next we observe that 
\begin{align*}
\frac{\zeta-\zeta^4}{\zeta^2-\zeta^3}
&=-(\zeta^2+\zeta^3) \\
\frac{\zeta^2-\zeta^3}{\zeta-\zeta^4}&=\zeta+\zeta^4 ,
\end{align*}
so that 
\begin{gather*}
 \zeta^j\frac{-(\zeta-\zeta^4)\zeta^{1-l}t +(\zeta^2-\zeta^3)}{(\zeta^2-\zeta^3)\zeta^{1-l}t+(\zeta-\zeta^4)}  =
  \zeta^j\frac{(\zeta^2+\zeta^3)t +\zeta^{l-1}}{t-\zeta^{l-1}(\zeta^2+\zeta^3)}\\
\zeta^j\frac{(\zeta^2-\zeta^3)\zeta^{1-l}t+(\zeta-\zeta^4)}
{(\zeta-\zeta^4)\zeta^{1-l}t -(\zeta^2-\zeta^3)}=\zeta^j\frac{(\zeta+\zeta^4)t+\zeta^{l-1}}
{ t -\zeta^{l-1}(\zeta+\zeta^4)}
\end{gather*}
for $j,l=1,\dots 5$.
Thus if we set 
$$
\beta(t):=\frac{(\zeta^2+\zeta^3)t +1}{t-(\zeta^2+\zeta^3)},
$$
then 
\begin{align*}
\phi^j\circ\beta\circ\phi^{1-l}(t) &= \zeta^j\frac{(\zeta^2+\zeta^3)t +\zeta^{l-1}}{t-\zeta^{l-1}(\zeta^2+\zeta^3)} , \\
\psi\circ\beta(t) 
&=  -\zeta^4\frac{(\zeta+\zeta^4)t+1} { t - (\zeta+\zeta^4)}
\end{align*}
and 
\begin{align*}
\beta(t)-\zeta^r(\zeta+\zeta^4)
&=\frac{(1-\zeta^r)(\zeta^2+\zeta^3) t+1+\zeta^r(\zeta^2+\zeta^3)^2}{t-(\zeta^2+\zeta^3)} 
\end{align*}
So we will only calculate $\beta$ evaluated on the basis elements
\begin{align*}
\beta(\omega_0)
&=\omega_1-\omega_{6} \\  
\beta(\omega_1)
&=\omega_0-\omega_{6} \\ 
\beta(\omega_2)
&=\omega_5-\omega_{6} \\   
\beta(\omega_3)
&=\omega_8-\omega_{6} 
\end{align*}

\begin{align*}
\beta(\omega_4)
&=\omega_9-\omega_{6} \\  
\beta(\omega_5)
&=\omega_2-\omega_{6} \\  
\beta(\omega_6)
&=-\omega_{6} \\ 
\beta(\omega_7)
&=\omega_{10}-\omega_6 \\ 
\end{align*}
\begin{align*}
\beta(\omega_8)
&=\omega_3-\omega_{6} \\  
\beta(\omega_9)
&=\omega_4-\omega_{6} \\  
\beta(\omega_{10})
&=\omega_7-\omega_6.  
\end{align*}
$$
\beta^2(\omega_{10})=\beta(\omega_7-\omega_6)=\omega_{10}
$$
The automorphism $\beta\circ \phi$ has order $3$ and thus is in the conjugacy class of $(123)$.  Moreover its trace is $-1$. 
The character table of $A_5$ is the following
\begin{table}[htp]
\caption{Character of $A_5$.}
\begin{center}
\begin{tabular}{|c|c|c|c|c|c|}
\hline
&I & $(123)$ & $(12)(34)$ & $(12345)$ & $(21345)$  \\
\hline
$\chi_{\Omega_{R_a}/dR_a}$& 11 & -1 & -1 &1  &  1 \\
$I$& 1 & 1 & 1 &1 & 1  \\
$\chi_{2}$& 3& 0& -1 & $\frac{1+\sqrt{5}}{2}$ & $\frac{1-\sqrt{5}}{2}$ \\
$\chi_{3}$& 3& 0&  -1&   $\frac{1-\sqrt{5}}{2}$ & $\frac{1+\sqrt{5}}{2}$ \\
$\chi_{4}$& 4& 1&  0& -1 & -1 \\
$\chi_{5}$& 5& -1& 1& 0 & 0  \\
\hline
\end{tabular}
\end{center}
\label{default}
\end{table}%

Solving 
\begin{align*}
11&=a+3b+3c+4d+5e \\
-1&=a+d-e \\
-1&=a-b-c +e \\
1&=a+\frac{1+\sqrt{5}}{2}b+\frac{1-\sqrt{5}}{2}c -d\\
1&=a+\frac{1-\sqrt{5}}{2}b+\frac{1+\sqrt{5}}{2}c -d\\
\end{align*}
we get $a=d=0$ and $b=c=e=1$.  That is to say $\chi_{\Omega_{R_a}/dR_a}=\chi_{2}+\chi_{3}+\chi_{5}$.

The conjugacy class of $\beta\circ \phi$ has $20$ elements in it.  The conjugacy class of the two cycle $\psi$ has $15$ elements in it. There are two conjugacy classes that have $5$-cycles in them and they both have 12 elements in them.

The matrices for the above linear transformations with respect to the ordered basis $\{\omega_0,\dots, \omega_{10}\}$ 
The linear transformations $\phi$, $\beta$ and $\psi$ have with respect to this basis have matrix representation
\begin{align*}
\phi&=\begin{pmatrix}
 1 & 0 & 0 & 0 & 0 & 0 & 0 & 0 & 0 & 0 & 0 \\
 0 & 0 & 1 & 0 & 0 & 0 & 0 & 0 & 0 & 0 & 0 \\
 0 & 0 & 0 & 1 & 0 & 0 & 0 & 0 & 0 & 0 & 0 \\
 0 & 0 & 0 & 0 & 1 & 0 & 0 & 0 & 0 & 0 & 0 \\
 0 & 0 & 0 & 0 & 0 & 1 & 0 & 0 & 0 & 0 & 0 \\
 0 & 1 & 0 & 0 & 0 & 0 & 0 & 0 & 0 & 0 & 0 \\
 0 & 0 & 0 & 0 & 0 & 0 & 0 & 1 & 0 & 0 & 0 \\
 0 & 0 & 0 & 0 & 0 & 0 & 0 & 0 & 1 & 0 & 0 \\
 0 & 0 & 0 & 0 & 0 & 0 & 0 & 0 & 0 & 1 & 0 \\
 0 & 0 & 0 & 0 & 0 & 0 & 0 & 0 & 0 & 0 & 1 \\
 0 & 0 & 0 & 0 & 0 & 0 & 1 & 0 & 0 & 0 & 0 \\
 \end{pmatrix},\\ 
 \psi&=
 \begin{pmatrix}
-1 & -1 & -1 & -1 & -1 & -1 & -1 & -1 & -1 & -1 & -1 \\
 0 & 0 & 0 & 0 & 0 & 0 & 0 & 0 & 0 & 0 & 1 \\
 0 & 0 & 0 & 0 & 0 & 0 & 0 & 0 & 0 & 1 & 0 \\
 0 & 0 & 0 & 0 & 0 & 0 & 0 & 0 & 1 & 0 & 0 \\
 0 & 0 & 0 & 0 & 0 & 0 & 0 & 1 & 0 & 0 & 0 \\
 0 & 0 & 0 & 0 & 0 & 0 & 1 & 0 & 0 & 0 & 0 \\
 0 & 0 & 0 & 0 & 0 & 1 & 0 & 0 & 0 & 0 & 0 \\
 0 & 0 & 0 & 0 & 1 & 0 & 0 & 0 & 0 & 0 & 0 \\
 0 & 0 & 0 & 1 & 0 & 0 & 0 & 0 & 0 & 0 & 0 \\
 0 & 0 & 1 & 0 & 0 & 0 & 0 & 0 & 0 & 0 & 0 \\
 0 & 1 & 0 & 0 & 0 & 0 & 0 & 0 & 0 & 0 & 0 \\
\end{pmatrix}  \\
\beta&=
\begin{pmatrix}
 0 & 1 & 0 & 0 & 0 & 0 & 0 & 0 & 0 & 0 & 0 \\
 1 & 0 & 0 & 0 & 0 & 0 & 0 & 0 & 0 & 0 & 0 \\
 0 & 0 & 0 & 0 & 0 & 1 & 0 & 0 & 0 & 0 & 0 \\
 0 & 0 & 0 & 0 & 0 & 0 & 0 & 0 & 1 & 0 & 0 \\
 0 & 0 & 0 & 0 & 0 & 0 & 0 & 0 & 0 & 1 & 0 \\
 0 & 0 & 1 & 0 & 0 & 0 & 0 & 0 & 0 & 0 & 0 \\
 -1 & -1 & -1 & -1 & -1 & -1 & -1 & -1 & -1 & -1 & -1 \\
 0 & 0 & 0 & 0 & 0 & 0 & 0 & 0 & 0 & 0 & 1 \\
 0 & 0 & 0 & 1 & 0 & 0 & 0 & 0 & 0 & 0 & 0 \\
 0 & 0 & 0 & 0 & 1 & 0 & 0 & 0 & 0 & 0 & 0 \\
 0 & 0 & 0 & 0 & 0 & 0 & 0 & 1 & 0 & 0 & 0 \\
\end{pmatrix}
\end{align*}

To get bases for the three irreducible subrepresentations we need to use the projection formulae and sum over the conjugacy classes (using Mathematica).
\begin{align*}
\pi_{\chi_5}&=\frac{1}{|A_5|}\sum_{g\in A_5}\chi_{5}(g)\rho(g) =\frac{1}{60}\sum_{g\in A_5}\chi_{5}(g)\rho(g)  \\
&=\frac{1}{60}\begin{pmatrix}
 0 & -6 & -6 & -6 & -6 & -6 & -6 & -6 & -6 & -6 & -6 \\
 0 & 6 & 0 & 0 & 0 & 0 & 6 & 0 & 0 & 0 & 0 \\
 0 & 0 & 6 & 0 & 0 & 0 & 0 & 6 & 0 & 0 & 0 \\
 0 & 0 & 0 & 6 & 0 & 0 & 0 & 0 & 6 & 0 & 0 \\
 0 & 0 & 0 & 0 & 6 & 0 & 0 & 0 & 0 & 6 & 0 \\
 0 & 0 & 0 & 0 & 0 & 6 & 0 & 0 & 0 & 0 & 6 \\
 0 & 6 & 0 & 0 & 0 & 0 & 6 & 0 & 0 & 0 & 0 \\
 0 & 0 & 6 & 0 & 0 & 0 & 0 & 6 & 0 & 0 & 0 \\
 0 & 0 & 0 & 6 & 0 & 0 & 0 & 0 & 6 & 0 & 0 \\
 0 & 0 & 0 & 0 & 6 & 0 & 0 & 0 & 0 & 6 & 0 \\
 0 & 0 & 0 & 0 & 0 & 6 & 0 & 0 & 0 & 0 & 6 \\
 \end{pmatrix}.
\end{align*}
\end{proof}
Thus a basis of the irreducible 5 dimensional subrepresentation is $\{-\omega_0+\omega_i+\omega_{i+5}\,|\,1\leq i\leq 5\}$.

For one of the two three dimensional irreducible subrepresentations we have projection operator:

\begin{align*}
\pi_{\chi_2}&=\frac{1}{|A_5|}\sum_{g\in A_5}\chi_{2}(g)\rho(g) =\frac{1}{60}\sum_{g\in A_5}\chi_{2}(g)\rho(g)  \\
&=\frac{\sqrt{5}}{30}\begin{pmatrix}
\sqrt{5} & \xi & \xi & \xi & \xi & \xi &-\bar\xi & 
-\bar\xi&-\bar\xi &-\bar\xi &-\bar\xi \\
 1 & \xi & 1 & 0 & 0 & 1 &
  \bar\xi & 0 & 1 & 1 & 0 \\
 1 & 1 & \xi & 1 & 0 & 0 & 0 &
  \bar\xi & 0 & 1 & 1 \\
 1 & 0 & 1 & \xi & 1 & 0 & 1 &
   0 &\bar\xi & 0 & 1 \\
 1 & 0 & 0 & 1 & \xi & 1 & 1 &
   1 & 0 &\bar\xi & 0 \\
 1 & 1 & 0 & 0 & 1 & \xi & 0 &
   1 & 1 & 0 &\bar\xi \\
 -1 &-\xi& -1 & 0 & 0 & -1 &
  -\bar\xi & 0 & -1 & -1 & 0 \\
 -1 & -1 &-\xi& -1 & 0 & 0 &
   0 &-\bar\xi & 0 & -1 & -1 \\
 -1 & 0 & -1 &-\xi& -1 & 0 &
   -1 & 0 &-\bar\xi & 0 & -1 \\
 -1 & 0 & 0 & -1 &-\xi& -1 &
   -1 & -1 & 0 &-\bar\xi & 0 \\
 -1 & -1 & 0 & 0 & -1 &-\xi&
   0 & -1 & -1 & 0 &-\bar\xi \\
   \end{pmatrix}.
\end{align*}
This means that a basis for $U_3$ is 
\begin{align*}
&\sqrt{5}\omega_0+\sum_{i=1}^5(\omega_i-\omega_{i+5}), \\
 &\xi \omega_0+\xi \omega_1+\omega_2+\omega_5-\xi\omega_6-\omega_7-\omega_{10}, \\
 &\xi \omega_0+ \omega_1+\xi\omega_2+\omega_3-\omega_6-\xi\omega_7-\omega_8.
\end{align*}
The projection operator for the other three dimensional subrepresentation is 
\begin{align*}
\pi_{\chi_3}&=\frac{1}{|A_5|}\sum_{g\in A_5}\chi_{3}(g)\rho(g) =\frac{1}{60}\sum_{g\in A_5}\chi_{3}(g)\rho(g)  \\
&=\frac{\sqrt{5}}{30}\begin{pmatrix}
\sqrt{5} &-\bar\xi & -\bar\xi & -\bar\xi & -\bar\xi & -\bar\xi & \xi & \xi & \xi & \xi & \xi \\
 -1 &-\bar\xi & -1 & 0 & 0 & -1 &
  -\xi & 0 & -1 & -1 & 0 \\
 -1 & -1 &-\bar\xi & -1 & 0 & 0 &
   0 &-\xi & 0 & -1 & -1 \\
 -1 & 0 & -1 &-\bar\xi & -1 & 0 &
   -1 & 0 &-\xi & 0 & -1 \\
 -1 & 0 & 0 & -1 &-\bar\xi & -1 &
   -1 & -1 & 0 &-\xi & 0 \\
 -1 & -1 & 0 & 0 & -1 &-\bar\xi &
   0 & -1 & -1 & 0 &-\xi \\
 1 &\bar\xi & 1 & 0 & 0 & 1 &\xi & 0 & 1 & 1 & 0 \\
 1 & 1 &\bar\xi & 1 & 0 & 0 & 0 &
  \xi & 0 & 1 & 1 \\
 1 & 0 & 1 &\bar\xi & 1 & 0 & 1 &
   0 &\xi & 0 & 1 \\
 1 & 0 & 0 & 1 &\bar\xi & 1 & 1 &
   1 & 0 &\xi & 0 \\
 1 & 1 & 0 & 0 & 1 &\bar\xi & 0 &
   1 & 1 & 0 &\xi \\   \end{pmatrix}.
\end{align*}
So a basis for $U_3'$ is 
\begin{align*}
&\sqrt{5}\omega_0-\sum_{i=1}^5(\omega_i-\omega_{i+5}), \\
 &-\bar\xi \omega_0-\bar\xi \omega_1-\omega_2-\omega_5+\bar\xi\omega_6+\omega_7+\omega_{10}, \\
 &-\bar\xi \omega_0- \omega_1-\bar\xi\omega_2-\omega_3+\omega_6+\bar\xi\omega_7+\omega_8.
\end{align*}

\section{Acknowledgements} This work was partially supported by a grant from the Simons Foundation (\#319261).  The author would also like to thank the Mittag-Leffler Institute and its organizers for its hospitality during part of the spring semester of 2015 where a portion of this work was done. We would also like to thank Kaiming Zhao for useful discussions and the referee for a list of relevant references. 

\section{Concluding Remark.}
There are other rings $R$ that arise as rings of meromorphic functions on a Riemann surface with a finite number of points removed.  We plan to investigate how $\Omega_{R}/dR$ decomposes under the action their automorphism group of $R$.    
\def\cprime{$'$} \def\cprime{$'$} \def\cprime{$'$}
\providecommand{\bysame}{\leavevmode\hbox to3em{\hrulefill}\thinspace}
\providecommand{\MR}{\relax\ifhmode\unskip\space\fi MR }
\providecommand{\MRhref}[2]{%
  \href{http://www.ams.org/mathscinet-getitem?mr=#1}{#2}
}
\providecommand{\href}[2]{#2}

\def\cprime{$'$} \def\cprime{$'$} \def\cprime{$'$}

 \end{document}